\documentclass[reqno]{amsart}

\usepackage{graphicx}
\usepackage[norelsize, ruled]{algorithm2e}

\newtheorem{theorem}{Theorem}[section]
\newtheorem{lemma}[theorem]{Lemma}
\newtheorem{corollary}[theorem]{Corollary}

\theoremstyle{definition}
\newtheorem{definition}[theorem]{Definition}

\newtheorem{proposition}[theorem]{Proposition}

\theoremstyle{remark}

\newcommand{\floor}[1]{\lfloor #1 \rfloor}

\numberwithin{equation}{section}

\begin{document}

\title{Partitions of Equiangular Tight Frames}

\author{James Rosado}
\address{Department of Mathematics, Rowan University} 
\curraddr{201 Mullica Hill Road,
Glassboro,  New Jersey 08028}
\email{rosado42@rowan.edu}

\author{Hieu D. Nguyen}
\address{Department of Mathematics, Rowan University} 
\curraddr{201 Mullica Hill Road,
Glassboro,  New Jersey 08028}
\email{nguyen@rowan.edu}


\author{Lei Cao}
\address{Department of Mathematics, Georgian Court University}
\curraddr{900 Lakewood Avenue, Lakewood, NJ 08701
}
\email{leicaomath@gmail.com}



\date{11-11-2016}


\begin{abstract}
We present a new efficient algorithm to construct partitions of a special class of equiangular tight frames (ETFs) that satisfy the operator norm bound established by a theorem of Marcus, Spielman, and Srivastava (MSS), which they proved as a corollary yields a positive solution to the Kadison-Singer problem. In particular, we prove that certain diagonal partitions of complex ETFs generated by recursive skew-symmetric conference matrices yield a refinement of the MSS bound.  Moreover, we prove that all partitions of ETFs whose largest subset has cardinality three or less also satisfy the MSS bound.
\end{abstract}

\maketitle


\section{Introduction}
The concept of frames is relatively new and has its origins in harmonic and functional analysis, operator theory, linear algebra, and matrix theory. Frames can be thought as a relaxation of a vector space basis. Unlike a basis, a frame consists of vectors that are not necessarily linearly independent; however, they still span the entire vector space.  Frames were born as a result of the limited capabilities of bases being utilized in signal processing and sampling theory. In particular, frames allow for redundancy, meaning a set of frame vectors may contain multiple copies of the same vector, and/or some of the vectors are linear combinations of other frame vectors. This property is useful in the fields of communications and signal processing where signals suffer from erasures or noise; redundant frame vectors can then be used to reconstruct the original signal \cite{KOVA}.
\par
One class of important frames are Grassmannian frames, i.e., those with low coherence where the maximum correlation between its vectors is minimized.  They have connections to packings in Grassmannian spaces, spherical codes, and strongly regular graphs.  Equiangular tight frames (ETFs) are Grassmanian frames which meet the Welch bound and represent maximal packings of lines in real or complex space.
\par
In this article we present a linear-time algorithm to construct partitions of a special class of complex ETFs and prove that the subset norms of these partitions, i.e., the norms of the sum of outer products of all the frame vectors in each subset, meet the operator norm bound specified by a theorem of Marcus, Spielman, and Srivastava (MSS) \cite{MSS2}.  Finding algorithms to partition frames in general was described as an open problem in \cite{MSS2}.  We believe that our algorithm is the first of its kind for any known class of frames and in particular for ETFs generated from skew-symmetric conference matrices. Such algorithms are useful because not every partition satisfies the MSS theorem; a counterexample is given after Theorem \ref{MSS}. On the other hand, we also prove that partitions of ETFs whose subsets all have cardinality less than four must satisfy the MSS bound. 
\par 
The proof of our algorithm relies on constructing special partitions that we call diagonal partitions since their subsets correspond to sub-block matrices along the diagonal of the Gram matrix associated to each EFT.  Fortunately, these Gram matrices, when constructed from skew-symmetric conference matrices, have spectra that can be computed exactly.  As a result, we obtain a bound on the subset norms of these diagonal partitions and show that it is sharper than the bound described by the following theorem \cite[Corollary 1.5]{MSS2}, which we call the MSS theorem.

\begin{theorem}[Marcus-Spielman-Srivastava \cite{MSS2}]\label{MSS}
Let $r$ be a positive integer and $\mathcal{V}=$ \\ $\{v_1,\ldots,v_n\}$ be a set of vectors in  $\mathbb{C}^m$ satisfying
\begin{equation}\label{MSSHYP}
\sum_{j=1}^n v_j\otimes v_j=\mathbf{I},
\end{equation}
where $\lVert v_j \rVert^2\leq \delta$ for all $j$. Then there exists a partition $\mathcal{P}=\{\mathcal{S}_1,\ldots,\mathcal{S}_r\}$ of $[n]=\{1,\ldots,n\}$ such that the subset norms satisfy
\begin{equation}\label{eqn18}
    \left\lVert\sum_{j\in \mathcal{S}_k}v_j\otimes v_j \right\rVert\leq \left(\frac{1}{\sqrt{r}}+\sqrt{\delta} \right)^2
\end{equation}
for all $k\in\{1,\ldots, r\}$.
\end{theorem}
Any partition that satisfies the bound in Theorem \ref{MSS} will be called a MSS partition. Numerical experiments show that partitions of ETFs with small subset cardinalities are MSS partitions. On the other hand not all partitions are MSS partitions; numerical calculations reveal that given an ETF with 32 vectors, a partition where one subset has cardinality 16 and all other subsets have cardinality 1 is not necessarily a MSS partition. An example of such an invalid partition is one that contains the subset $\mathcal{S}=\{2,5,7,8,9,10,11,13,14,16,17,19,21,25,27,31\}$, where the values in this subset correspond to the indices of the vectors in the frame $\mathcal{V}=\mathcal{F}=\{f_j\}_{j=1}^{32}$ as constructed in Section 2. If we compute the subset norm of $\mathcal{S}$, then we find that its value is greater than the MSS bound for $r=17$.
\par
Our main result is the following theorem, which states that partitions of a certain form, called diagonal partions, are in fact MSS partitions and provides a bound on their subset norms that is sharper than (\ref{eqn18}) with $\delta=1/2$.
\begin{theorem}\label{JHL} 
Let $\mathcal{F}=\{f_j\}_{j=1}^{2^k}$ be an ETF constructed from $\mathbf{R}= \mathbf{I}+i\alpha \mathbf{C}$ as described in Corollary \ref{NEWGRASS},  where $\alpha=1/\sqrt{2^k-1}$ and $\mathbf{C}$ is a  skew-symmetric conference matrix defined recursively by (\ref{eqn13}). Fix $r$ to be a positive integer. Then for all diagonal partitions $\mathcal{P}=\{\mathcal{S}_1,\ldots,\mathcal{S}_r\}$ of the index set $\mathcal{I}=\{1,\ldots,2^k\}$ with $|S_h| \leq 2^{k-\lfloor \log_2 r \rfloor}$ for all $h\in \{1,\ldots, r\}$, the subset norms satisfy
\begin{equation}
    \left\lVert\sum_{j\in\mathcal{S}_h}f_j\otimes f_j\right\rVert\leq \frac{1}{2}+\frac{1}{\sqrt{2r}}
\end{equation}
for all $h\in\{1,\ldots,r\}$.
\end{theorem}
The rest of the paper is organized as follows.  In Section 2 we formally state definitions and theorems that are employed in our analysis. In Section 3 we derive the norm formulas for a particular Gram matrix called the $\mathbf{R}$ matrix. In Sections 4 and 5 we relate norms formed from subsets of our frame to the norms of sub-blocks of the $\mathbf{R}$ matrix. In Section 6 we prove Theorem \ref{JHL} our algorithm and describe our algorithm for constructing diagonal partitions.  Lastly, in Section 7 we prove that all partitions whose subsets have cardinalities less than four are automatically MSS partitions for any ETF.

\section{Preliminaries}
This section is organized as follows: first we shall describe the lexicon that is utilized in our proofs and analysis. Then in the second subsection we define frames and describe an important class called Grassmannian frames. The third and last subsection is dedicated to describing the construction of equiangular tight frames.
\subsection{Notation}
Throughout this article matrices will be denoted by bold capital letters, e.g.,  $\mathbf{M}\in\mathbb{R}^{n\times n}$ denotes an $n\times n$ matrix. Vectors are denoted with lower case letters, e.g., $x\in\mathbb{C}^{n}$ denotes an $n$-dimensional complex vector. The standard notation for transpose and conjugation will be utilized on matrices and vectors, e.g., $\overline{\mathbf{M}}$, $\mathbf{M}^T$, and $\mathbf{M}^*$ indicate the conjugation, transpose, and conjugate transpose of $\mathbf{M}$, respectively.

A set will be denoted with a calligraphic capital letter, e.g., $\mathcal{S}=\{1,2,3\}$, and $\lvert\mathcal{S}\rvert$ refers to its cardinality.  Given a set of vectors $\mathcal{V}=\{v_k\}_{k=1}^{n}$, we define $\mathcal{V}_{\mathcal{S}}=\{v_k\in\mathcal{V}:k\in\mathcal{S}\}$. For example, if $\mathcal{S}=\{1,2,3\}$, then  $\mathcal{V}_{\mathcal{S}}=\{v_1,v_2,v_3\}$.  A partition of $\mathcal{S}$ will be denoted by $\mathcal{P}=\{\mathcal{S}_1,\mathcal{S}_2,\ldots\mathcal{S}_k\}$.

Next, we define the Euclidean norm of a 
vector $x\in\mathbb{C}^n$ by
\begin{equation}\label{eqn2}
    \lVert x\rVert=\sqrt{x^*x}.
\end{equation}
The induced matrix (or operator) norm of a matrix $\mathbf{M} \in \mathbb{C}^{m\times n}$ is defined by 
\begin{equation}\label{eqn3}
    \lVert \mathbf{M}\rVert = \max_{\lVert x\rVert=1}\lVert\mathbf{M}x\rVert
\end{equation}
for $x\in\mathbb{C}^n$.
Recall that when $\mathbf{M}$ is a Hermitian positive semidefinite matrix, then $\lVert \mathbf{M}\rVert$ is given by the largest eigenvalue of $\mathbf{M}$.
Given two column vectors $x,y\in\mathbb{C}^n$, the inner product of $x$ and $y$ is defined to be
\[
\langle x,y\rangle=x^*y
\]
and the tensor (or outer) product is defined to be
\begin{equation}\label{eqn4}
    x\otimes y=xy^*.
\end{equation}
Observe that $xy^*\in\mathbb{C}^{n\times n}$.

\begin{definition}[Subset Norm]  Let $\mathcal{F}=\{f_1,\ldots,f_n\}$ be a set of vectors and $\mathcal{S}$ a subset of $[n]=\{1,\ldots n\}$.  We define the {\em subset outer product} of $\mathcal{S}$ to be matrix
\[
\mathbf{F}_\mathcal{S}=\sum_{j\in \mathcal{S}} f_j\otimes f_j
\]
and the {\em subset norm} of $\mathcal{S}$ to be the induced matrix norm $\lVert \mathbf{F}_\mathcal{S} \rVert$.
\end{definition}

\subsection{Frames} In this section we shall discuss important definitions involving general frames as described in \cite{CK,OLE1, OLE2, HKLW}.
\begin{definition}\label{FRAME}
A frame for a Hilbert space $\mathbb{H}$ is a sequence of vectors $\mathcal{F}=\{f_j\}\subset\mathbb{H}$ for which there exists constants $0<A\leq B<\infty$ such that for every vector $x\in\mathbb{H}$,
\begin{equation}\label{eqn7}
    A\lVert x\rVert^2\leq \sum_j|\langle x,f_j\rangle|^2\leq B\lVert x\rVert^2.
\end{equation}
\end{definition}
In this paper we shall assume our Hilbert space to be $\mathbb{R}^n$ or $\mathbb{C}^n$.  Moreover, we shall only be considering finite frames, and as a result of Definition \ref{FRAME} the frame vectors span the Hilbert space, i.e.,
\begin{equation}\label{eqn8}
    \mathrm{span} \ \mathcal{F}=\mathbb{R}^n \ \mathrm{or} \ \mathbb{C}^n.
\end{equation}

Next, we define the \textit{maximal frame correlation} of a frame as described in \cite{GRSS}.
\begin{definition}\label{MAXCORR}
For a given unit norm frame $\mathcal{F}=\{f_j\}_{j=1}^n$ in $\mathbb{H}^m$ we define the maximal frame correlation $\mathcal{M}(\mathcal{F})$ by
\begin{equation}\label{eqn9}
    \mathcal{M}(\mathcal{F})=\max_{j,k,k\neq j}\{| \langle f_j,f_k\rangle|\}.
\end{equation}
\end{definition}

We now give the formal definition of Grassmannian frames \cite{GRSS}.
\begin{definition}\label{GRASS}
A sequence of vectors $\{u_j\}_{j=1}^n$ in $\mathbb{H}^m$ is called Grassmannian if it is the solution to
\begin{equation}\label{eqn10}
    \min\{\mathcal{M}(\mathcal{F})\},
\end{equation}
where the minimum is taken over all unit norm frames $\mathcal{F}$ in $\mathbb{H}^m$.
\end{definition}

Strohmer and Heath \cite{GRSS} proved the following lower bound for $\mathcal{M}(\mathcal{F})$ and also proved that this bound is achieved for 
equiangular tight frames (ETFs), i.e,
 $\mathcal{F}=\{f_j\}_{j=1}^n$ where
\[
    |\langle f_j,f_l \rangle|^2= c \textrm{ for all }j,l\textrm{ with } j\neq l
\]
for some constant $c\geq 0$.

\begin{theorem}[\cite{GRSS}]
Let $\{f_k\}_{k=1}^n$ be a frame for $\mathbb{H}^m$. Then
\begin{equation}\label{eq:welch-bound}
     \mathcal{M}(\{f_k\}_{k=1}^n)\geq\sqrt{\frac{n-m}{m(n-1)}}
\end{equation}
and equality holds if and only if $\{f_k\}_{k=1}^n$ is an equiangular tight frame (ETF) with
\[
    |\langle f_j,f_l \rangle|^2= \frac{n-m}{m(n-1)}
\]
for all $1\leq j,l \leq n$ with $j\neq l$.  Furthermore, equality can only hold in (\ref{eq:welch-bound}) if $n\leq m(m+1)/2$ when $\mathbb{H}=\mathbb{R}$ and $n\leq m^2$ when $\mathbb{H}=\mathbb{C}$.
\end{theorem}

Inequality (\ref{eq:welch-bound}) is known as the Welch bound. Thus, ETFs are exactly those Grassmannian frames which meet the Welch bound (also called optimal Grassmannian frames in the context of \cite{GRSS} if all the frame vectors have unit norm).

\subsection{Construction of Equiangular Tight Frames} 
Strohmer and Heath \cite{GRSS} give the following construction of ETFs.

\begin{corollary}[\cite{GRSS}]\label{COR24}
 Let $m,n\in\mathbb{N}$ with $n\geq m$. Assume $\mathbf{R}$ is a Hermitian $n\times n$ matrix in $\mathbb{H}$  with entries $\mathbf{R}_{k,k}=1$ and
\begin{equation}\label{eq:R-matrix}
    \mathbf{R}_{k,l}=\left\lbrace\begin{array}{cc}
          \pm\sqrt{\frac{n-m}{m(n-1)}} & \textit{ if } \mathbb{H}=\mathbb{R},  \\[4pt]
         \pm i \sqrt{\frac{n-m}{m(n-1)}} & \textit{ if } \mathbb{H}=\mathbb{C},
    \end{array}\right.
\end{equation}
for $k,l=1,\ldots,n$ with $k\neq l$. If the eigenvalues $\lambda_1,\ldots,\lambda_{n}$ of $\mathbf{R}$ are such that $\lambda_1=\ldots=\lambda_m=\frac{n}{m}$ and $\lambda_{m+1}=\ldots=\lambda_{n}=0$, then there exists a frame $\{f_k\}_{k=1}^n$ in $\mathbb{H}^{m}$ that achieves the Welch bound (\ref{eq:welch-bound}), i.e., $\mathcal{F}=\{f_k\}_{k=1}^n$ is a unit-norm ETF.
\end{corollary}

To construct unit-norm ETFs from (\ref{eq:R-matrix}), Strohmer and Heath consider the spectral decomposition of $\mathbf{R}$:
\begin{equation*}
    \mathbf{R}= \mathbf{W\Lambda W}^*,
\end{equation*}
where $\mathbf{\Lambda}$ is the diagonal matrix consisting of eigenvalues of $\mathbf{R}$ and the columns of $\mathbf{W}$ are the eigenvectors of $\mathbf{R}$.
Then the first $m$ values on the diagonal of $\mathbf{\Lambda}$ consist of $2$'s and the remaining $n-m$ consist of $0$'s. Next, we define
\begin{equation} \label{eq:rescale}
    f_j=\sqrt{\frac{n}{m}}\{\mathbf{W}_{j,l}\}_{l=1}^m.
\end{equation}
In other words, $f_j$ consists of the first $m$ entries of the $j$th-row of $\mathbf{W}$. Since $\langle f_j,f_l\rangle=R_{l,j}$, this shows that $\mathcal{F}=\{f_j\}_{j=1}^n$ is a unit-norm ETF.

However, unit-norm ETFs do not meet the assumptions of the MSS theorem.  The following result,
\begin{align*}
    \mathbf{F}:=\sum_{j=1}^{n}f_j\otimes f_j&=\sum_{j=1}^{2^k}\left(\sqrt{\frac{n}{m}}\{\mathbf{W}_{j,l}\}_{l=1}^{m}\right)\left(\sqrt{\frac{n}{m}}\{\mathbf{W}_{j,l}\}_{l=1}^{m}\right)^*=\frac{n}{m}\mathbf{I}_m,
\end{align*}
shows that the frame vectors of $\mathcal{F}$ should not be scaled by a factor of $\sqrt{n/m}$ in (\ref{eq:rescale}) if the condition $\mathbf{F}=\mathbf{I}_m$ is to hold, where $\mathbf{I}_m$ is the $m\times m$ identity matrix.  We summarize this in the following corollary.

\begin{corollary}\label{NEWGRASS}
Let $\mathbf{R}$ be defined as in Corollary \ref{COR24}. Then the set of vectors $\mathcal{F}=\{f_l\}_{l=1}^{n}$, where
\begin{equation}\label{NEWFRAME}
    f_j=\{\mathbf{W}_{j,l}\}_{l=1}^m \in \mathbb{H}^m
\end{equation}
for $j\in\{1,\ldots,n\}$ and $l\in\{1,\ldots,m\}$,  is an ETF that satisfies the hypothesis of the MSS theorem, namely, $\mathbf{F}=\mathbf{I}_m$.
\end{corollary}
As a result of this corollary, we have
\begin{equation}\label{INNERPROD}
 f_j^* f_l=\frac{m}{n}\mathbf{R}_{j,l}=\left\lbrace\begin{array}{cc}
          \frac{m}{n} & j=l;  \\[4pt]
         \pm c\frac{m}{n}{\sqrt{\frac{n-m}{m(n-1)}}} & j\neq l,
    \end{array}\right.
\end{equation}
where $c=1$ if $\mathbb{H}=\mathbb{R}$ and $c=i$ if $\mathbb{H}=\mathbb{C}$.

Next, we review Strohmer and Heath's construction of $\mathbf{R}$ via conference matrices by assuming that it takes the form
\begin{equation}\label{eqn14}
    \mathbf{R}=\mathbf{I}+i\alpha\mathbf{C},
\end{equation}
where $\alpha$ is a constant and $\mathbf{C}$ is conference matrix, defined by Geothals and Seidel \cite{CONF} as follows.
\begin{definition}\label{CMATRIX}
A matrix $\mathbf{C}$ of order $n$ with diagonal elements $0$ and off diagonal elements $\pm 1$ satisfying
\begin{equation}\label{eqn11}
    \mathbf{CC}^T=(n-1)\mathbf{I}
\end{equation}
is said to be a conference matrix.
\end{definition}

\section{Spectrum of the $\mathbf{R}$-matrix}
In this section we shall find the spectrum and the norm of the $\mathbf{R}$-matrix, but first we need to find the spectrum and the norm of conference matrices.
\begin{proposition}
The norm of a conference matrix $\mathbf{C}$ of order $n$ is given by
\begin{equation}\label{CONFNORM}
    \lVert\mathbf{C}\rVert=\sqrt{n-1}.
\end{equation}
\end{proposition}
\begin{proof}
This follows easily from (\ref{eqn11}): 
\begin{align*}
    \lVert \mathbf{C}\rVert&=\max_{\lVert x\rVert=1}\lVert \mathbf{C}x\rVert=\max_{\lVert x\rVert=1}\sqrt{\lVert \mathbf{C}x\rVert^2}=\max_{\lVert x\rVert=1}\sqrt{x^T\mathbf{C}^T\mathbf{C}x}=\sqrt{n-1}.
\end{align*}
\end{proof}
Next we shall determine the spectrum of our conference matrices in the case where $\mathbf{C}$ is skew-symmetric, i.e., $\mathbf{C}=-\mathbf{C}^T$.  One recursive method of constructing such matrices of order $n=2^k$ is described in \cite{GRSS} as follows. 
We begin with 
\begin{equation*}
    \mathbf{C}(1)=
    \begin{bmatrix}
    0 & -1 \\ 1 & 0 
    \end{bmatrix}
\end{equation*}
and define $\mathbf{C}(k)$ recursively by
\begin{equation}\label{eqn13}
    \mathbf{C}(k)=
    \begin{bmatrix}
     \mathbf{C}(k-1) & \mathbf{C}(k-1)-\mathbf{I}_{2^{k-1}} \\
     \mathbf{C}(k-1)+\mathbf{I}_{2^{k-1}} & -\mathbf{C}(k-1)
    \end{bmatrix}.
\end{equation}
It is straightforward to prove by induction that $\mathbf{C}(k)$ is in fact a skew-symmetric conference matrix.

Conference matrices are special cases of Seidel matrices where their entries come from $\{0,\pm 1\}$.  Greaves \cite{GREAVES} proves the following theorem regarding the characteristic polynomial of Seidel matrices.
\begin{theorem}[\cite{GREAVES}]\label{SEIDEL}
Let $\mathbf{S}$ be a Seidel matrix and $\chi_\mathbf{S}(x)$ be its characteristic polynomial. Then
\begin{equation}\label{CHAR}
\chi_\mathbf{S}(x)=(x^2+4t+3)^{2t+2}\Leftrightarrow (4t+3)\mathbf{I}+\mathbf{S}^2=\mathbf{0}, 
\end{equation}
where $\mathbf{0}$ is a matrix of zeros.
\end{theorem}
With this theorem at hand, we have the following result, which we provide a proof for completeness since we were unable to a find proof of it in the literature.

\begin{proposition}\label{CONFSPEC}
The spectrum $\sigma(\mathbf{C})$ of a skew-symmetric conference matrix $\mathbf{C}$ of order $n$, where $n$ is an even positive integer, is given by 
\begin{equation*}
    \sigma(\mathbf{C})=\{\lambda_1,\lambda_2\}=\{\pm i\sqrt{n-1}\}
\end{equation*}
and each spectral value has multiplicity $n/2$.
\end{proposition}

\begin{proof}
It follows from skew-symmetry that
\begin{align*}
    (n-1)\mathbf{I}+\mathbf{C}^2
    &=(n-1)\mathbf{I}-\mathbf{CC}^T=(n-1)\mathbf{I}-(n-1)\mathbf{I}=\mathbf{0}.
\end{align*}
It follows from Theorem \ref{SEIDEL} that the characteristic polynomial of $\mathbf{C}$ is given by (\ref{CHAR}) with $4t+3=n-1$ and $2t+2=n/2$, or equivalently,
\begin{equation*}
    (x^2+n-1)^{n/2}=0.
\end{equation*}
The solution set to this equation is $\{\pm i\sqrt{n-1}\}$ and each solution has multiplicity $n/2$. Thus,
\begin{equation*}
    \sigma(\mathbf{C})=\{\lambda_1,\lambda_2\}=\{\pm i\sqrt{n-1}\}
\end{equation*}
and each spectral value has multiplicity $n/2$.
\end{proof}

Proposition \ref{CONFSPEC} now allows us to compute the spectrum of $\mathbf{R}$.

\begin{proposition}\label{RSPEC}
The spectrum of the matrix $\mathbf{R}=\mathbf{I}+i\alpha\mathbf{C}$, where $\mathbf{C}$ is a skew symmetric conference matrix of even order $n$ and $\alpha$ is a real constant, is given by
\begin{equation}\label{rspecEQN}
    \sigma(\mathbf{R})=\{1\pm\alpha\sqrt{n-1}\},
\end{equation}
where each spectral value has multiplicity $n/2$.
\end{proposition}
\begin{proof}
To find the spectrum of $\mathbf{R}=\mathbf{I}+i\alpha\mathbf{C}$ we shall solve the following characteristic equation:
\begin{align*}
    \text{det}( \mathbf{R}-\lambda\mathbf{I})=\text{det}( i\alpha\mathbf{C}-(\lambda-1)\mathbf{I})=0.
\end{align*}
But observe that if we make the substitution $\gamma=\lambda-1$, then our characteristic equation becomes
\begin{equation*}
    \text{det}( i\alpha\mathbf{C}-\gamma\mathbf{I}) = 0.
\end{equation*}
which is the characteristic equation for $i\alpha\mathbf{C}$.  It follows from Proposition \ref{CONFSPEC} that
\begin{equation*}
    \sigma(i\alpha\mathbf{C})=i\alpha\cdot\sigma(\mathbf{C})=\{\pm\alpha\sqrt{n-1}\}.
\end{equation*}
Hence, $\sigma(\mathbf{R})=\{1\pm\alpha\sqrt{n-1}\}$.
\end{proof}

Knowing the spectrum of $\mathbf{R}$ allows us to immediately obtain the norm $\mathbf{R}$.
\begin{proposition}\label{pr:norm-R}
The norm of the matrix $\mathbf{R}=\mathbf{I}+i\alpha\mathbf{C}$, where $\mathbf{C}$ is a $n\times n$ skew symmetric conference matrix of order $n=2^k$ and $\alpha$ is a real constant, is given by
\[
\lVert \mathbf{R}\rVert=1+\lvert\alpha\rvert\sqrt{n-1}.
\]
\end{proposition}

Lastly, observe that if we set
\begin{equation}\label{eqn15}
    \alpha=\sqrt{\frac{n-m}{m(n-1)}},
\end{equation}
then it follows from Proposition \ref{RSPEC} that $\mathbf{R}=\mathbf{I}+i\alpha \mathbf{C}$, where $\mathbf{C}$ is defined by (\ref{eqn13}) with $n=2m$, satisfies Corollary \ref{COR24}.  Thus, Corollary \ref{NEWGRASS} yields an ETF $\mathcal{F}$ over $\mathbb{C}^m$ that satisfies the hypothesis of the MSS theorem.

\section{Analysis of $\mathbf{R}$-Matrix Sub-Blocks}
In this section we shall examine the sub-blocks of the $\mathbf{R}$-matrix that are along its main diagonal.  We assume that $\mathbf{R}(k)=\mathbf{I}+i\alpha\mathbf{C}(k)$ where $\mathbf{C}(k)$ is a recursively defined skew-symmetric conference matrix of order $2^k$ defined by (\ref{eqn13}).  Then $\mathbf{R}(k)$ can be expressed in block form
\begin{equation}\label{eqn21}
    \mathbf{R}(k)    =\begin{bmatrix}
    \mathbf{I}+i\alpha\mathbf{C}(k-1) & \star \\[8pt]
    \star & \mathbf{I}-i\alpha\mathbf{C}(k-1)
    \end{bmatrix}=\begin{bmatrix}
    \mathbf{R}(k-1) & \star \\[8pt]
    \star & \overline{\mathbf{R}}(k-1)
    \end{bmatrix}.
\end{equation}

Notice that we can apply this recursion again to get
\begin{equation*}
    \mathbf{R}(k)=\begin{bmatrix}
    \mathbf{R}_{2,1}(k) & \star & \star & \star \\[8pt]
    \star & \mathbf{R}_{2,2}(k) & \star & \star \\[8pt]
    \star & \star & \mathbf{R}_{2,3}(k) & \star \\[8pt]
    \star & \star & \star & \mathbf{R}_{2,4}(k)
    \end{bmatrix},
\end{equation*}
where
$\mathbf{R}_{2,q}(k)=\mathbf{R}(k-2)$ for $q=1,4$ and $\mathbf{R}_{2,q}(k)=\overline{\mathbf{R}}(k-2)$ for $q=2,3$.
We shall refer to the sub-blocks $\mathbf{R}_{2,q}(k)$ as \textit{recursive diagonal sub-blocks} of $\mathbf{R}(k)$ at depth 2. This generalizes to the following result for sub-blocks $\mathbf{R}_{d,q}(k)$ at depth $d$.
\begin{proposition}\label{RSUBS}
At depth $d\in\{1,\ldots,k\}$ , we have
\begin{equation}\label{eqn22}
    \mathbf{R}_{d,q}(k)=\mathbf{R}(k-d) \text{ or } \mathbf{R}_{d,q}(k)=\overline{\mathbf{R}}(k-d).
\end{equation}
for $q=1,\ldots,2^d$.
\end{proposition}
\begin{proof}
We shall prove (\ref{eqn22}) by induction.

\noindent Base case: If $d=1$, then it is clear from (\ref{eqn21}) that $\mathbf{R}_{1,q}(k)=\mathbf{R}(k-1)$ or $\mathbf{R}_{1,q}(k)=\overline{\mathbf{R}}(k-1)$.

\noindent Inductive step: Let $b\in\{1,\ldots,k-1\}$ be given and suppose (\ref{eqn22}) is true for $d=b$ so that
\begin{equation*}
    \mathbf{R}(k)=\begin{bmatrix}
    \mathbf{R}_{b,1}(k) & \star & \cdots & \\[8pt]
    \star & \overline{\mathbf{R}}_{b,2}(k) &\cdots & \\[8pt]
    \vdots &\vdots &\ddots  
    \end{bmatrix},
\end{equation*}
where $\mathbf{R}_{b,q}=\mathbf{R}(k-b)$ or $\mathbf{R}_{b,q}=\overline{\mathbf{R}}(k-b)$ for $q=1,\ldots,2^b$.  But each sub-block $\mathbf{R}_{b,q}$ can be partitioned into smaller sub-blocks of the form 
\begin{equation}
    \mathbf{R}_{b,q}(k)=
    \begin{bmatrix}
    \mathbf{R}_{b+1,q'}(k) & \star \\[8pt]
    \star & \mathbf{R}_{b+1,q''}(k)
    \end{bmatrix}.
\end{equation}
Thus we have $\mathbf{R}(k)=$
\begin{equation*}
\small
    \begin{bmatrix}
    \mathbf{R}_{b+1,1}(k) & \star & \star & \star & \cdots & \\[8pt]
    \star & \mathbf{R}_{b+1,2}(k) & \star & \star & \cdots & \\[8pt]
    \star & \star & \mathbf{R}_{b+1,3}(k) & \star & \cdots  & \\[8pt]
    \star & \star & \star & \mathbf{R}_{b+1,4}(k) & \cdots  & \\[8pt]
    \vdots & \vdots & \vdots & \vdots & \ddots &
    \end{bmatrix},
\end{equation*}
where there are twice as many diagonal sub-blocks compared to depth $b$. Moreover,
\begin{align*}
    \mathbf{R}_{b+1,q}(k)=\mathbf{R}(k-(b+1))\text{ or }\mathbf{R}_{b+1,q}(k)=\overline{\mathbf{R}}(k-(b+1));
\end{align*}
This completes the proof.
\end{proof}

The following lemma gives the norm of our sub-blocks $\mathbf{R}_{d,q}$, which follows immediately from Proposition \ref{pr:norm-R}.
\begin{lemma}\label{THM42}
Let $\mathbf{R}(k)=\mathbf{I}+i\alpha\mathbf{C}(k)$ where $\mathbf{C}(k)$ is a recursively defined skew-symmetric conference matrix of order $2^k$ and $\alpha=1/\sqrt{2^k-1}$. Then the norm of each recursive diagonal sub-block $\mathbf{R}_{d,q}(k)$ at depth $d$ is 
\begin{equation}
    \lVert \mathbf{R}_{d,q}(k)\rVert=1+\alpha\sqrt{2^{k-d}-1}.
\end{equation}
\end{lemma}

\par
It remains to characterize the positions of the entries of $\mathbf{R}(k)$ that appear in the recursive sub-blocks $\mathbf{R}_{d,q}(k)$. This is given by the following lemma, which follows easily from the recursive definition of our sub-blocks. 

\begin{lemma}
Let $\mathbf{R}(k)=\mathbf{I}+i\alpha\mathbf{C}(k)$ where $\mathbf{C}(k)$ is a skew-symmetric conference matrix of order $2^k$ defined by (\ref{eqn13}).  Then the entries of the $q$-th diagonal sub-block $\mathbf{R}_{d,q}(k)$ at recursive depth $d$ are given by $(\mathbf{R}_{d,q}(k))_{g,h}=f_h^* f_g$, where $(g,h)\in\mathcal{S}_{d,q}\times\mathcal{S}_{d,q}$ 
and $S_{d,q}$ is defined by
\begin{equation}\label{eqn25}
    \mathcal{S}_{d,q}=\{(q-1)\cdot2^{k-d}+1,\ldots, q\cdot2^{k-d}\},
\end{equation}
for $q\in\{1,\ldots,2^d\}$ and $d\in\{0,\ldots,k-1\}$.
\end{lemma}

\section{Subsets of $[2^k]$}
Let $\mathcal{F}=\{f_j\}_{j=1}^n$ be an ETF derived from Corollary \ref{NEWGRASS} and let $\mathcal{I}=[n]=\{1,\ldots,n\}$ be the index set for $\mathcal{F}$. Suppose $\mathcal{S}=\{j_r\}_{r=1}^q\subset\mathcal{I}$ is a $q$-element subset of $\mathcal{I}$ and $\mathcal{F}_S=\{f_{j_r}\}_{r=1}^q$ a linearly independent set of frame vectors.  As before, we define
\begin{align*}
    \mathbf{F}_{\mathcal{S}}&=\sum_{i\in\mathcal{S}}f_i\otimes f_i=\sum_{r=1}^q f_{j_r} f_{j_r}^*.
\end{align*}
To find the norm $\lVert \mathbf{F}_{\mathcal{S}}\rVert$, we shall solve the eigenvalue problem
\begin{equation}\label{eq:eigenvalue-problem}
    \mathbf{F}_{\mathcal{S}}v=\lambda v,
\end{equation}
where $v\in \mathbb{C}^m$.  Assume $\lambda\neq 0$.  Then since $\mathbf{F}_{\mathcal{S}} v \in \mathrm{span}(\mathcal{F}_{\mathcal{S}})$ and $\mathcal{F}_{\mathcal{S}}$ is a linearly independent set, we may assume that the $v \in \mathrm{span}(\mathcal{F}_{\mathcal{S}})$, i.e., 
\begin{equation*}
    v=\sum_{r=1}^q c_{j_r}f_{j_r}
\end{equation*}
for some set of coefficients $\{c_1,\ldots,c_q\}$.
Equation (\ref{eq:eigenvalue-problem}) becomes
\begin{align*}
    \left(\sum_{r=1}^q f_{j_r} f_{j_r}^*\right)\left(\sum_{r=1}^q c_{j_r}f_{j_r}\right) &=\lambda\sum_{r=1}^q c_{j_r}f_{j_r}\\[8pt]
    \Rightarrow \sum_{r=1}^q\left(\sum_{h=1}^q c_{j_h}f_{j_r}^*f_{j_h} \right)f_{j_r} &=\lambda\sum_{r=1}^q c_{j_r}f_{j_r}.
\end{align*}
Then equating coefficients yields $q$ equations in $q$ unknowns $c_1,\ldots,c_q$:
\begin{equation*}
    \sum_{h=1}^q c_{j_h}f_{j_r}^*f_{j_h}=\lambda c_{j_r}, \ \ r=1,\ldots, q,
\end{equation*}
which we express in matrix form:
\begin{equation} \label{eq:gram-matrix}
    \begin{bmatrix}
    f_{j_1}^*f_{j_1} & f_{j_1}^*f_{j_2} & f_{j_1}^*f_{j_3} &\cdots & f_{j_1}^*f_{j_q} \\[4pt]
    f_{j_2}^*f_{j_1} & f_{j_2}^*f_{j_2} & f_{j_2}^*f_{j_3} &\cdots & f_{j_2}^*f_{j_q} \\[4pt]
    f_{j_3}^*f_{j_1} & f_{j_3}^*f_{j_2} & f_{j_3}^*f_{j_3} &\cdots & f_{j_3}^*f_{j_q} \\[4pt]
    \vdots & \vdots & \vdots & \ddots & \vdots \\[4pt]
    f_{j_q}^*f_{j_1} & f_{j_q}^*f_{j_2} & f_{j_q}^*f_{j_3} &\cdots & f_{j_q}^*f_{j_q}
    \end{bmatrix}
    \begin{bmatrix}
    c_{j_1} \\[4pt]
    c_{j_2} \\[4pt]
    c_{j_3} \\[4pt]
    \vdots \\
    c_{j_q}
    \end{bmatrix}=\lambda \begin{bmatrix}
    c_{j_1} \\[4pt]
    c_{j_2} \\[4pt]
    c_{j_3} \\[4pt]
    \vdots \\
    c_{j_q}
    \end{bmatrix}.
\end{equation}
Notice that if we view $\mathcal{F}_\mathcal{S}$ as a matrix whose columns are its frame vectors, then the matrix on the left hand side of (\ref{eq:gram-matrix}) is the Gram matrix $\mathcal{F}_{\mathcal{S}}^*\mathcal{F}_{\mathcal{S}}$ whose entries are defined by (\ref{INNERPROD}). In fact, this Gram matrix is a principal sub-matrix of $\mathbf{R}$. Let us denote by $\mathbf{R}_S$ the principal sub-matrix of $\mathbf{R}$ whose row and column indices come from $\mathcal{S}$ and denote by $c_{\mathcal{S}}$ the column vector whose entries are the coefficients $\{c_{j_r}\}_{r=1}^q$. Then $\frac{m}{n}\mathbf{R}_S=\mathcal{F}_S^*\mathcal{F}_S$ and equation (\ref{eq:gram-matrix}) becomes
\begin{equation} \label{eq:eigenvalue-R}
    \frac{m}{n}\mathbf{R}_{\mathcal{S}}c_{\mathcal{S}}=\lambda c_\mathcal{S}.
\end{equation}

Since the eigenvalue problems (\ref{eq:eigenvalue-problem}) and (\ref{eq:eigenvalue-R}) are equivalent for all nonzero $\lambda$,  it follows that $\sigma(\mathbf{F}_{\mathcal{S}})=\sigma(\mathbf{R}_{\mathcal{S}}) \cup \{0\}$. This proves the following lemma.

\begin{lemma} \label{le:norm-subset}
Let $\mathcal{F}=\{f_j\}_{j=1}^{2^k}$ be an ETF defined by Corollary \ref{NEWGRASS} and $\mathcal{S}\subset\mathcal{I}$ an arbitrary subset with $\mathcal{F}_{\mathcal{S}}$ a linearly independent set of frame vectors.
Then $\sigma(\mathbf{F}_{\mathcal{S}})=\sigma(\mathbf{R}_{\mathcal{S}})\cup \{0\}$ and thus
\begin{equation*}
    \lVert\mathbf{F}_\mathcal{S}\rVert=\frac{m}{n}\lVert\mathbf{R}_\mathcal{S}\rVert,
\end{equation*}
where the entries of $\mathbf{R}_\mathcal{S}$ correspond to the entries $(g,h)\in\mathcal{S}\times\mathcal{S}$ of $\mathbf{R}$. 
\end{lemma}

Now that we know how the subset norm of $\mathcal{F}_\mathcal{S}$  relates to the norm of the corresponding sub-block $\mathbf{R}_\mathcal{S}$ of $\mathbf{R}$, let us next consider frames $\mathcal{F}$ defined by $\mathbf{R}=\mathbf{I}+i\alpha\mathbf{C}$, where $\mathbf{C}$ is a $n\times n$ skew symmetric conference matrix of order $n=2m=2^k$ and $\alpha=1/\sqrt{2^k-1}$.  Moreover, we restrict our attention to subsets that correspond to diagonal sub-blocks of $\mathbf{R}$.  In particular, we have the following result.

\begin{lemma} \label{le:norm-sub-block}
Suppose $\mathcal{S}_{d,q}\subset\mathcal{I}$, where $\mathcal{S}_{d,q}$ is defined by (\ref{eqn25}). Then
\begin{equation}\label{eq:norm-sub-block}
\lVert\mathbf{F}_{\mathcal{S}_{d,q}}\rVert=\frac{1}{2}\lVert\mathbf{R}_{d,q}(k)\rVert.
\end{equation}
\end{lemma}

\begin{proof}
From Proposition \ref{RSPEC}, we have that 
\[
\sigma(\mathbf{R}_{d,q}(k))=\left\{1\pm \frac{\sqrt{2^{k-d}-1}}{\sqrt{2^k-1}}\right\}.
\]
In particular, all eigenvalues of $\mathbf{R}_{d,q}(k)$ are nonzero for $d\geq 1$, which implies that $\mathbf{R}_{d,q}(k)$ has full rank.  It follows that $\mathcal{F}_{\mathcal{S}}$ (viewed as matrix of column vectors) also has full rank since $\frac{1}{2}\mathbf{R}_{d,q}(k)=\mathcal{F}_{{\mathcal{S}}_{d,q}}^*\mathcal{F}_{{\mathcal{S}}_{d,q}}$.  Thus, $\mathcal{F}_{\mathcal{S}}$ is a linearly independent set and (\ref{eq:norm-sub-block}) now follows from Lemma \ref{le:norm-subset}.
\end{proof}

The next theorem follows immediately from Lemmas \ref{THM42} and \ref{le:norm-sub-block}.
\begin{theorem} \label{th:norm-sub-block}
Let $\mathcal{F}=\{f_j\}_{j=1}^{2^k}$ be an ETF defined by Corollary \ref{NEWGRASS}, where $\mathbf{R}(k)=\mathbf{I}+i\alpha\mathbf{C}(k)$ with $\mathbf{C}(k)$ is a skew-symmetric conference matrix of order $n=2m=2^k$ and $\alpha=1/\sqrt{2^k-1}$. Then the norm of $\mathbf{F}_{\mathcal{S}_{d,q}}=\sum_{j\in\mathcal{S}_{d,q}}f_j\otimes f_j$, where $\mathcal{S}_{d,q}\subset\mathcal{I}=[2^k]$ is a subset corresponding to $q$-th recursive diagonal sub-block $\mathbf{R}_{d,q}(k)$ at depth $d$ of $\mathbf{R}(k)$, is given by
\begin{equation}\label{eqn27}
    \lVert\mathbf{F}_{\mathcal{S}_{d,q}}\rVert=\frac{1}{2}+\frac{1}{2}\frac{\sqrt{2^{k-d}-1}}{\sqrt{2^k-1}},
\end{equation}
where $d\in\{0,\ldots,k-1\}$.   
\end{theorem}
Notice that if $d=0$ in Theorem \ref{th:norm-sub-block}, then this corresponds to a subset of size one, and the norm reduces to exactly 1; this makes sense because $\mathcal{S}_{0,q}=\mathcal{I}$ in which case $\mathbf{F}_\mathcal{S}=\mathbf{I}$. On the other hand, if $d=k$ then $\mathcal{S}_{k,q}$ consists of a single element in which case it is easy to show that $\mathbf{F}_\mathcal{S}=\frac{1}{2}$ for a single frame vector.

\section{Diagonal Partitions and the MSS Theorem}
The MSS theorem guarantees that any ETF  which meets the hypothesis (\ref{MSSHYP}) can be partitioned to meet the norm bound  (\ref{eqn18}). In this section, we describe partitions of ETF described by (\ref{NEWGRASS}) that satisfy the MSS theorem. Towards this end, let
\begin{equation*}
    \mathcal{P}=\{\mathcal{S}_1,\ldots,\mathcal{S}_r\}
\end{equation*}
denote a partition of $\mathcal{I}=[2^k]$ into disjoint subsets. Thus, $\mathcal{I}=\bigcup_{h=1}^r\mathcal{S}_h$. If we assume each subset $\mathcal{S}_h$ to be a diagonal sub-block defined by (\ref{eqn25}), i.e., 
\begin{equation*}
    \mathcal{P}=\{\mathcal{S}_{d_1,q_1},\ldots,\mathcal{S}_{d_r,q_r}\},
\end{equation*}
then we shall call $\mathcal{P}$ a \textit{diagonal partition}. As an example suppose we form a partition $\mathcal{P}$ of a frame $\mathcal{F}=\{f_j\}_{j=1}^{16}$ consisting of 16 vectors into $r=8$ subsets as follows: 
\begin{align*}
\mathcal{P} &= \{\mathcal{S}_{3,1},\mathcal{S}_{3,2},\mathcal{S}_{2,2},\mathcal{S}_{3,5},\mathcal{S}_{3,6},\mathcal{S}_{3,7},\mathcal{S}_{4,15},\mathcal{S}_{4,16}\} \\
&=\{\{1,2\},\{3,4\},\{5,6,7,8\},\{9,10\},\{11,12\},\{13,14\},\{15\},\{16\}\}.   
\end{align*}
Then $\mathcal{P}$ is a diagonal partition because the indices in each subset of $\mathcal{P}$ correspond to entries at the row/column locations of the diagonal sub-blocks of $\mathbf{R}$. Figure \ref{fig1} shows visually these diagonal sub-blocks.
\begin{figure}[ht]
\centering
\includegraphics[scale=0.5]{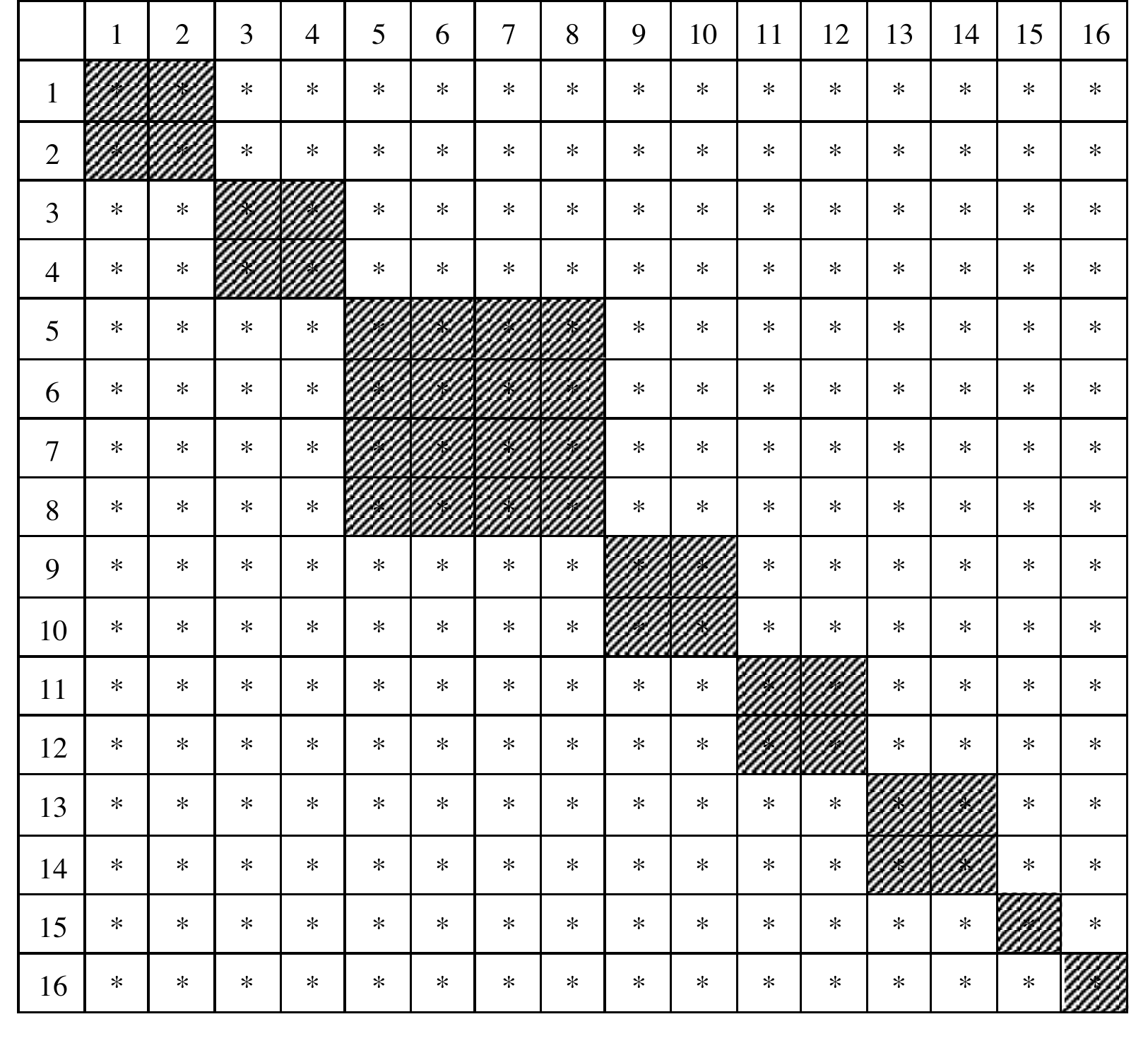}
\caption{Diagonal Partition $\mathcal{P}$ of $\mathcal{F}=\{f_j\}_{j=1}^{16}$}
\label{fig1}
\end{figure}
Notice that the cardinality of each subset is given by $\lvert \mathcal{S}_{d_h,q_h} \rvert=2^{k-d_h}$ and that each subset consists of consecutive elements. 

We are finally ready to prove Theorem \ref{JHL} by showing that certain diagonal partitions are MSS partitions and that the bound on their subset norms derived from our proof is sharper than the MSS bound.
\begin{proof}[Proof of Theorem \ref{JHL}]
Given a diagonal partition $\mathcal{P}$ of $\mathcal{I}$ consisting of $r$ subsets, there must exists at least one subset $\mathcal{S}_{max}\in\mathcal{P}$ such that
\begin{equation*}
    \lvert\mathcal{S}_{d_h,q_h}\rvert \leq \lvert\mathcal{S}_{max}\rvert
\end{equation*}
for all $h\in\{1,\ldots,r\}$. In other words, in $\mathcal{P}$ there is one subset that has largest (or equal) cardinality compared to all other subsets. We focus on this largest subset because the norm described by (\ref{eqn27}) increases with the cardinality of the subset $\mathcal{S}_{d,q}$. It follows that
\[
\lVert\mathbf{F}_{\mathcal{S}_{d_h,q_h}}\rVert \leq \lVert\mathbf{F}_{\mathcal{S}_{max}}\rVert.
\]
Hence, we need only be concerned about bounding the subset norm of $\lvert\mathcal{S}_{max}\rvert$. 

Next, observe that $\lvert \mathcal{S}_{max}\rvert=2^{k-d_{min}}$, where $d_{min}=\min\{d_1,\ldots,d_r\}$. It follows that
\begin{equation}\label{eqn40}
    \lVert \mathbf{F}_{\mathcal{S}_{max}} \rVert=\frac{1}{2}+\frac{1}{2}\frac{\sqrt{2^{k-d_{min}}-1}}{\sqrt{2^k-1}}.
\end{equation}
Moreover, we have
\begin{equation}\label{eqn41}
    2^{d_{min}}\leq r \leq 2^{d_{min}+1}.
\end{equation}
We now manipulate the right hand side of (\ref{eqn40}) to obtain
\begin{align}\label{eqn42}
    \frac{1}{2}+\frac{1}{2}\frac{\sqrt{2^{k-d_{min}}-1}}{\sqrt{2^k-1}}&=\frac{1}{2}+\frac{1}{2}\frac{\sqrt{2^{k-d_{min}}-1}}{\sqrt{2^{d_{min}}}\sqrt{2^{k-d_{min}}-2^{-d_{min}}}}\notag \\[4pt]
    &\leq \frac{1}{2}+\frac{1}{\sqrt{2^{d_{min}+2}}} 
\end{align}
It follows from (\ref{eqn41}) that
\begin{equation}\label{eqn43}
    \frac{1}{\sqrt{2^{d_{min}+2}}}\leq\frac{1}{\sqrt{2r}}.
\end{equation}
Then combining (\ref{eqn40}), (\ref{eqn42}), and (\ref{eqn43})  yields
\begin{equation*}
     \lVert\mathbf{F}_{\mathcal{S}_{d_h,q_h}}\rVert \leq\frac{1}{2}+\frac{1}{\sqrt{2r}}
\end{equation*}
as desired.
\end{proof}

Since the inequalities
\begin{equation*}
    \lVert\mathbf{F}_{\mathcal{S}_{max}}\rVert\leq\frac{1}{2}+\frac{1}{\sqrt{2r}}\leq\left(\frac{1}{\sqrt{r}}+\frac{1}{\sqrt{2}}\right)^2
\end{equation*}
hold, it follows that any diagonal partition described in Theorem \ref{JHL} is an MSS partition with $\delta=1/2$, but with a sharper bound. Table \ref{tab:my_label} compares the two bounds for various values of $r$.
\begin{table}[h]
    \centering
    \begin{tabular}{|c|l|c|}
     \hline\hline
         $r$ & \text{Thm. \ref{JHL}} & \text{MSS Thm.} \\
         \hline
         & {\small $\frac{1}{2}+\frac{1}{\sqrt{2}\sqrt{r}}$}&{\small $\left(\frac{1}{\sqrt{r}}+\frac{1}{\sqrt{2}} \right)^2$} \\
         \hline\hline
 1 & 1.20711 & 2.91421 \\
 2 & 1. & 2. \\
 3 & 0.908248 & 1.64983 \\
 4 & 0.853553 & 1.45711 \\
 5 & 0.816228 & 1.33246 \\
 6 & 0.788675 & 1.24402 \\
 7 & 0.767261 & 1.17738 \\
 8 & 0.75 & 1.125 \\
 \hline\hline
    \end{tabular}
    \caption{Comparison of bounds: Theorem \ref{JHL} vs. MSS Theorem}
    \label{tab:my_label}
\end{table}\\
Next, we shall describe an efficient algorithm to construct a diagonal partition $\mathcal{P}$ as described in Theorem \ref{JHL}.\\ \\
\begin{algorithm}[H]\label{ALG}
\SetAlgoLined
\KwResult{$\mathcal{I}=\{1,\ldots,2^k\}$  is partitioned into $\mathcal{P}=\{\mathcal{S}_{d_i,q_i}\}_{i=1}^r$.}
 Given $r>1$\;
 Initialize $d=\floor{\log_2 r}$, $r'=2^{d}$\;
 Initialize $\mathcal{P}=\emptyset$\;
 \For{$i=1$ \emph{\KwTo} $2(r-r')$}{
 $\mathcal{S}_{d+1,i}=\{(i-1)\cdot2^{k-(d+1)}+1,\ldots, i\cdot2^{k-(d+1)}\}$\;
 $\mathcal{P}=\mathcal{P}\cup\{\mathcal{S}_{d+1,i}\}$\;
 }
 \For{$i=r-r'+1$ \emph{\KwTo} $r'$}{
 $\mathcal{S}_{d,i}=\{(i-1)\cdot2^{k-d}+1,\ldots, i\cdot2^{k-d}\}$\;
 $\mathcal{P}=\mathcal{P}\cup\{\mathcal{S}_{d,i}\}$\;
 }
 \caption{Diagonal Partition Algorithm}
\end{algorithm}
\vspace{10pt}
 Algorithm \ref{ALG} begins by initializing $r'$ to be the closest power of $2$ less than $r$ and $\mathcal{P}$ is an empty collection of sets. The first loop will then form $2(r-r')$ subsets, with cardinality $2^{k-(d+1)}$, of consecutive elements from $\mathcal{I}$ and the first subset will begin with $1$. Notice that this first loop will form subsets from the first
\begin{align}\label{eqnAA}
    2(r-r')\cdot 2^{k-(d+1)}&=(r-2^{d})\cdot 2^{k-d}=r\cdot 2^{k-d}-2^k
\end{align}
elements of $\mathcal{I}$. Then the subsets are placed in the partition $\mathcal{P}$. 

The second loop will form $2r'-r$ subsets of cardinality $2^{k-d}$ from the remaining
\begin{align}\label{eqnBB}
    (2r'-r)\cdot 2^{k-d}=(2^{d+1}-r)\cdot 2^{k-d}=2^{k+1}-r\cdot 2^{k-d}
\end{align}
elements of $\mathcal{I}$. Notice that if we sum (\ref{eqnAA}) and (\ref{eqnBB}) we obtain $2^k$, which equals the total number of elements of $\mathcal{I}$.

\section{Partitions with Small Subsets}
In this section we prove more general results regarding partitions of arbitrary ETFs but whose subsets are restricted to cardinalities of three or less. First, we consider the case where our partition $\mathcal{P}$ is such that its largest subset has cardinality two. We prove that all such partitions are MSS partitions by exactly computing their corresponding subset norms.
\begin{lemma} \label{le:two-subset}
Let $\mathcal{F}=\{f_j\}_{j=1}^{n}$ be an ETF  constructed from the $\mathbf{R}$-matrix in Corollary \ref{NEWGRASS}. Then the norm $\lVert\mathbf{F}_{\mathcal{S}}\rVert$, where $\mathcal{S}$ is a two-element subset, is given by
\begin{equation}\label{eqn60}
\lVert \mathbf{F}_{\mathcal{S}}\rVert=\frac{m}{n}\left(1+ \sqrt{\frac{n-m}{m(n-1)}}\right).
\end{equation}
\end{lemma}
\begin{proof}
Without loss of generality, let $\mathcal{S}=\{1,2\}$. We claim that $\mathcal{F}_{\mathcal{S}}=\{f_1,f_2\}$ is a linearly independent set.  To prove this, we use (\ref{eq:R-matrix}) to calculate
\begin{align*}
    \det\mathbf{R}_{\mathcal{S}} & =
    \det 
    \left(
    \begin{array}{cc}
    1 & \pm c\sqrt{\frac{n-m}{m(n-1)}}\\
    \mp \bar{c}\sqrt{\frac{n-m}{m(n-1)}} & 1
    \end{array}
    \right) \\
    & = 1+\frac{n-m}{m(n-1)}
\end{align*}
where $c=1$ if $\mathbb{H}=\mathbb{R}$ and $c=i$ if $\mathbb{H}=\mathbb{C}$.  This shows $\det\mathbf{R}_{\mathcal{S}}>0$, which implies $\mathbf{R}_{\mathcal{S}}$ has full rank.  Thus, $\mathcal{F}_{\mathcal{S}}=\{f_1,f_2\}$ is a linearly independent set.

We now compute the norm of
\begin{equation*}
    \mathbf{F}_{\mathcal{S}}= f_1\otimes f_1+f_2\otimes f_2.
\end{equation*}
Towards this end, we shall find the spectrum of $\mathbf{F}_{\mathcal{S}}$ by solving the eigenvalue problem $\mathbf{F}_{\mathcal{S}}v=\lambda v$, where we assume $\lambda\neq 0$ and $v$ is a linear combination of $f_1$ and $f_2$ since $\mathbf{F}_{\mathcal{S}}v \in \mathrm{span}(\mathcal{F}_{\mathcal{S}})$, namely
\begin{equation*}
    v=c_1 f_1 + c_2 f_2. 
\end{equation*}
Then
\begin{align*}
    \mathbf{F}_{\mathcal{S}}v=\lambda v&\Rightarrow (f_1\otimes f_1+f_2\otimes f_2)(c_1 f_1 + c_2 f_2)=\lambda (c_1 f_1 + c_2 f_2)\\[4pt]
    &\Rightarrow (f_1 f_1^*+f_2 f_2^*)(c_1 f_1 + c_2 f_2)=\lambda (c_1 f_1 + c_2 f_2)\\[4pt]
    &\Rightarrow c_1 f_1 f_1^*f_1+c_1 f_2 f_2^*f_1
    +c_2 f_1 f_1^*f_2+c_2 f_2 f_2^*f_2=\lambda c_1 f_1 +\lambda c_2 f_2 \\[4pt]
    &\Rightarrow
    f_1(\beta_{1,1} c_1+\beta_{1,2}c_2) + f_2(\beta_{2,1}c_1+\beta_{2,2}c_2)=\lambda c_1 f_1 +\lambda c_2 f_2
\end{align*}
where $\beta_{j,j}=f_j^*f_j$ and $\beta_{j,k}=f_j^*f_k$.
We can equate the parts of the last equation above to get
\begin{align}\label{LEQN2S}
    \beta_{1,1} c_1+\beta_{1,2} c_2&=\lambda c_1\\
    \beta_{2,1}c_1+\beta_{2,2}c_2&= \lambda c_2.\notag
\end{align}
Dividing the two above equations and cross-multiplying yields
\begin{equation}\label{eq:c1c2}
    \beta_{1,1}c_1 c_2 +\beta_{1,2}c_2^2=\beta_{2,1}c_1^2+\beta_{2,2}c_1 c_2.
\end{equation}

In the case $\mathbb{H}=\mathbb{C}$, we have $\beta_{1,1}=\beta_{2,2}$ and $\beta_{1,2}=-\beta_{2,1}$ since $\mathbf{R}$ is Hermitian.  This simplifies (\ref{eq:c1c2}) to
\begin{equation*}
    c_2^2=-c_1^2.
\end{equation*}
Thus,
\begin{equation*}
    c_2=\pm i c_1.
\end{equation*}
If we substitute $c_2$ into (\ref{LEQN2S}), solve for $\lambda$, and substitute the exact values for $\beta_{1,1}$ and $\beta_{1,2}$ from (\ref{INNERPROD}), we obtain two non-zero eigenvalues:
\begin{equation*}
    \lambda=\beta_{1,1}\pm i\beta_{1,2}=\frac{m}{n}\pm \frac{m}{n}\sqrt{\frac{n-m}{m(n-1)}}.
\end{equation*}
Thus, 
\[
\lVert \mathbf{F}_{\mathcal{S}}\rVert=\lambda_{max}=\frac{m}{n}\left(1+ \sqrt{\frac{n-m}{m(n-1)}}\right).
\]

In the case $\mathbb{H}=\mathbb{R}$, we have instead $\beta_{1,2}=\beta_{2,1}$, in which case a similar derivation shows that the same formula holds for $\lVert \mathbf{F}_{\mathcal{S}}\rVert$.
\end{proof}
As a corollary we have the following result.

\begin{theorem}\label{subset2}
Let $\mathcal{F}$ be an ETF of size $n$ constructed from the matrix $\mathbf{R}$ defined in Corollary \ref{COR24}. Suppose $\mathcal{P}=\{\mathcal{S}_1,\ldots,\mathcal{S}_r\}$ is a partition $[n]$, where $\lvert \mathcal{S}_h\rvert\leq 2$ for every $h\in[r]$. Then the norm $\lVert \mathbf{F}_{\mathcal{S}_h}\rVert$ satisfies the inequality
\begin{equation}\label{BND2}
    \lVert \mathbf{F}_{\mathcal{S}_h}\rVert\leq \delta +\sqrt{\frac{\delta}{r}}
\end{equation}
for every $h\in[r]$ where $\delta=m/n$.
\end{theorem}
\begin{proof}
It is easy to prove that if $|\mathcal{S}_h|=1$, then $\lVert \mathbf{F}_{\mathcal{S}_h}\rVert = m/n=\delta$ and thus (\ref{BND2}) clearly holds.   Therefore, we assume $|\mathcal{S}_h|=2$.  Then since $r \leq n$, it follows from Lemma \ref{le:two-subset} that
\begin{align*}
\lVert \mathbf{F}_{\mathcal{S}}\rVert & =  \frac{m}{n}\left(1+ \sqrt{\frac{n-m}{m(n-1)}}\right) \\
& \leq \frac{m}{n}\left(1+ \sqrt{\frac{1}{m}}\right)
= \frac{m}{n}+\frac{1}{\sqrt{n}}\cdot \sqrt{\frac{m}{n}} \\
& \leq \delta +\sqrt{\frac{\delta}{r}}.
\end{align*}
\end{proof}
Observe that the norm bound (\ref{BND2}) is  sharper than both the MSS bound and the bound in Theorem \ref{JHL} where $\delta=1/2$. 

\par
Next we consider subsets of size three and obtain an analogous formula for their corresponding subset norms.
\begin{lemma}\label{le:subset-size-three}
Let $\mathcal{F}=\{f_j\}_{j=1}^{n}$ be an ETF  constructed from the $\mathbf{R}$-matrix in Corollary \ref{COR24}.  Then the norm $\lVert\mathbf{F}_{\mathcal{S}}\rVert$, where $\mathcal{S}$ is a three-element subset, is given by
\begin{equation}\label{eqn62}
    \lVert \mathbf{F}_{\mathcal{S}}\rVert=
    \left\lbrace \begin{array}{cc}
    \frac{m}{n}\left(1+ \epsilon \sqrt{\frac{(n-m)}{m(n-1)}}\right) 
    & \textit{ if } \mathbb{H}=\mathbb{R};  \\[4pt]
    \frac{m}{n}\left(1+ \sqrt{\frac{3(n-m)}{m(n-1)}}\right) 
    & \textit{ if } \mathbb{H}=\mathbb{C}, 
    \end{array}
    \right.
\end{equation}
where $\epsilon=1$ or 2 .
\end{lemma}
\begin{proof}
Again, without loss of generality, let $\mathcal{S}=\{1,2,3\}$.  We first show that $\mathcal{F}_{\mathcal{S}}$ is a linearly independent set for $m\geq 3$.  Set $\alpha=\sqrt{(n-m)/(m(n-1))}$.  Then using (\ref{eq:R-matrix}), we have
\begin{align*}
    \det \mathbf{R}_{\mathcal{S}}& =
    \det 
    \left(
    \begin{array}{ccc}
    1 & \pm c\alpha & \pm c\alpha \\
    \mp \bar{c}\alpha & 1 & \pm c\alpha \\
    \mp \bar{c}\alpha & \mp \bar{c}\alpha & 1
    \end{array}
    \right) \\
    & = 1+\alpha^2
    \pm \bar{c}\alpha
    \left(
    c\alpha\pm \alpha^2
    \right) 
    \pm
    \bar{c}\alpha
    \left(
    \alpha^2\pm c\alpha
    \right) \\ 
    & = 1+\alpha^2\pm \alpha^2 \pm \alpha^2 \pm
    \bar{c}\alpha^3 \pm \bar{c}\alpha^3
\end{align*}
where $c=1$ if $\mathbb{H}=\mathbb{R}$ and $c=i$ if $\mathbb{H}=\mathbb{C}$, and each entry in the upper-half of the diagonal of $\mathbf{R}_{\mathcal{S}}$ can take on any choice of signs ($\pm$).
It follows that
\begin{align*}
    \det \mathbf{R}_{\mathcal{S}} & \geq 1-\alpha^2 -2\alpha^3 > 1-\frac{3}{m}\geq 0
\end{align*}
for $m\geq 3$. Thus, $\mathcal{F}_{\mathcal{S}}$ is a linearly independent set.
To compute
\begin{equation*}
    \lVert \mathbf{F}_{\mathcal{S}}\rVert=\lVert f_1\otimes f_1+f_2\otimes f_2+f_3\otimes f_3\rVert,
\end{equation*}
we again solve the eigenvalue problem $\mathbf{F}_{\mathcal{S}}v=\lambda v$, where we assume $\lambda\neq 0$ and
\begin{equation*}
    v=c_1 f_1 + c_2 f_2 + c_3 f_3
\end{equation*}
since $\mathbf{F}_{\mathcal{S}}v \in \mathrm{span}(\mathcal{F}_{\mathcal{S}})$.  Then
\begin{align*}
    \mathbf{F}_{\mathcal{S}}v &=
    f_1(\beta_{1,1} c_1+\beta_{1,2} c_2+\beta_{1,3}c_3)
    +f_2(\beta_{2,1}c_1+\beta_{2,2}c_2+\beta_{2,3}c_3)\\
    & \ \ \ \ +f_3(\beta_{3,1}c_1+\beta_{3,2}c_2+\beta_{3,3}c_3),
\end{align*}
where $\beta_{j,j}=f_j^*f_j$ and $\beta_{j,k}=f_j^*f_k$. We equate coefficients from $\mathbf{F}_{\mathcal{S}}v=\lambda v$ to obtain
\begin{align}\label{LEQNS}
    \beta_{1,1} c_1+\beta_{1,2} c_2+\beta_{1,3}c_3&=\lambda c_1\\
    \beta_{2,1}c_1+\beta_{2,2}c_2+\beta_{2,3}c_3 &= \lambda c_2\notag\\
    \beta_{3,1}c_1+\beta_{3,2}c_2+\beta_{3,3}c_3 &= \lambda c_3.\notag
\end{align}
We then combine pairs of equations to eliminate $\lambda$, which results in
\begin{align*}
    \beta_{1,1}c_1 c_2 +\beta_{1,2}c_2^2+\beta_{1,3}c_3 c_2&=\beta_{2,1}c_1^2+\beta_{2,2}c_2 c_1+\beta_{2,3}c_3 c_1\\
    \beta_{2,1}c_1 c_3+\beta_{2,2}c_2 c_3+\beta_{2,3}c_3^2 &= \beta_{3,1}c_1 c_2+\beta_{3,2}c_2^2+\beta_{3,3}c_3 c_2\\
    \beta_{1,1} c_1 c_3+\beta_{1,2} c_2 c_3+\beta_{1,3}c_3^2 &= \beta_{3,1}c_1^2+\beta_{3,2}c_2 c_1+\beta_{3,3}c_3 c_1
\end{align*}
Then using the relations among the coefficient $\beta_{j,l}$ due to $\mathbf{R}$ being a Hermitian matrix, we find there are three solutions for $c_1,c_2,c_3$.  In the case where $\mathbb{H}=\mathbb{C}$, these solutions take the form
\begin{equation*}
    c_2 = -\frac{\beta_{1,3}}{\beta_{2,3}}c_1, c_3 = -\frac{\beta_{1,2}}{\beta_{2,3}}c_1,
\end{equation*}
and
\begin{align*}
    c_2&=c_1\frac{\beta_{1,3}\beta_{2,3}\mp\sqrt{-\beta_{1,2}^2\left(\beta_{1,2}^2+\beta_{1,3}^2+\beta_{2,3}^2\right)}}{\beta_{1,2}^2+\beta_{1,3}^2},\\[4pt]
    c_3&=\frac{c_1}{\beta_{1,2}}\frac{\beta_{2,3}\beta_{1,2}^2\pm\beta_{1,3}\sqrt{-\beta_{1,2}^2\left(\beta_{1,2}^2+\beta_{1,3}^2+\beta_{2,3}^2\right)}}{\beta_{1,2}^2+\beta_{1,3}^2}.
\end{align*}
We take one of our original equations in (\ref{LEQNS}) and isolate $\lambda$ to get
\begin{equation*}
    \lambda = \beta_{1,1},\beta_{1,1} \pm i\sqrt{\beta_{1,2}^2+\beta_{1,3}^2+\beta_{2,3}^2}.
\end{equation*}
Finally, we substitute values for $\beta_{j,l}$ from formula (\ref{INNERPROD}) to obtain three non-zero eigenvalues:
\begin{equation*}
    \lambda = \frac{m}{n},\frac{m}{n}\left(1 \pm \sqrt{\frac{3(n-m)}{m(n-1)}}\right).
\end{equation*}
Thus,
\begin{equation*}
     \lVert \mathbf{F}_{\mathcal{S}}\rVert=\lambda_{max}=\frac{m}{n}\left(1+ \sqrt{\frac{3(n-m)}{m(n-1)}}\right).
\end{equation*}

In the case $\mathbb{H}=\mathbb{R}$, we find that the three solutions for $c_1$, $c_2$, $c_3$ are of two types:
\begin{equation}
\mathrm{Type \ I}:
\begin{array}{ll}
c_1=0, & c_2=-c_3; \\
c_2=\pm c_1, & c_3=\pm c_1; \\
c_2=\mp \frac{c_1}{2}, & c_3=\mp \frac{c_1}{2}
\end{array}
\end{equation}
or
\begin{equation}
\mathrm{Type \ II}:
\begin{array}{ll}
c_1=0, & c_2=c_3; \\
c_2=\pm c_1, & c_3=\mp c_1; \\
c_2=\mp \frac{c_1}{2}, & c_3=\pm \frac{c_1}{2}.
\end{array}
\end{equation}
It follows that there are two distinct non-zero eigenvalues for $\lambda$ for all possible choice of signs in the solutions for $c_1$, $c_2$, and $c_3$.  Namely, they take the form
\begin{equation*}
    \lambda=\frac{m}{n}\left(1 \pm \sqrt{\frac{n-m}{m(n-1)}}\right), \frac{m}{n}\left(1 \mp 2\sqrt{\frac{n-m}{m(n-1)}}\right),
\end{equation*}
where the first eigenvalue listed has multiplicity two.
Thus,
\begin{equation*}
     \lVert \mathbf{F}_{\mathcal{S}}\rVert=\lambda_{max}=\frac{m}{n}\left(1+\epsilon \sqrt{\frac{(n-m)}{m(n-1)}}\right),
\end{equation*}
where $\epsilon=1$ or 2.
\end{proof}

Lemma \ref{le:subset-size-three} leads to a similar result on the norm bound for subsets of size three as previously obtained for subsets of size two.
\begin{theorem}\label{subset3}
Let $\mathcal{F}$ be an ETF of size $n$ constructed by the $\mathbf{R}$-matrix in Corollary \ref{COR24}. Suppose $\mathcal{P}=\{\mathcal{S}_1,\ldots,\mathcal{S}_r\}$ is a partition of $[n]$ with $\lvert \mathcal{S}_h\rvert\leq 3$ for every $h\in[r]$.  Then the norm $\lVert \mathbf{F}_{\mathcal{S}_h}\rVert$ satisfies the following inequality
\begin{equation}\label{eqn74}
    \lVert \mathbf{F}_{\mathcal{S}_h}\rVert\leq
    \left\lbrace
    \begin{array}{cc}
    \delta +\epsilon \sqrt{\frac{\delta}{r}} & \textit{ if } \mathbb{H}=\mathbb{R}; \\ [4pt]
    \delta +\sqrt{\frac{3\delta}{r}} & \textit{ if } \mathbb{H}=\mathbb{C}
    \end{array}
    \right.
\end{equation}
for every $h\in[r]$ where $\delta=m/n$ and $\epsilon=1$ or 2.
\end{theorem}
\begin{proof} If suffices to prove for cardinality equal to 3 since the inequality is true for subsets of cardinality less than or equal to 2 from Theorem \ref{subset2}.  We therefore assume $|\mathcal{S}_h|=3$. Then since $r \leq n$, it follows from Lemma \ref{le:subset-size-three} for the case $\mathbb{H}=\mathbb{R}$ that
\begin{align*}
\lVert \mathbf{F}_{\mathcal{S}}\rVert & =  \frac{m}{n}\left(1+ \epsilon\sqrt{\frac{(n-m)}{m(n-1)}}\right) \\
& \leq \frac{m}{n}\left(1+ \epsilon \sqrt{\frac{1}{m}}\right)
= \frac{m}{n}+\frac{\epsilon}{\sqrt{n}}\cdot \sqrt{\frac{m}{n}} \\
& \leq \delta +\epsilon \sqrt{\frac{\delta}{r}}
\end{align*}
A similar derivation holds for the case $\mathbb{H}=\mathbb{C}$, which we omit.
\end{proof}

Observe that we have the same implication as we did for a subset of cardinality two. The norm for a subset of cardinality three will always be sharper than MSS bound, although not as sharp compared to the subset of size 2. Figure \ref{fig2} gives a comparison of all subset norm bounds discussed in this paper.  

\begin{figure}[h]
\centering
\includegraphics[scale=0.5]{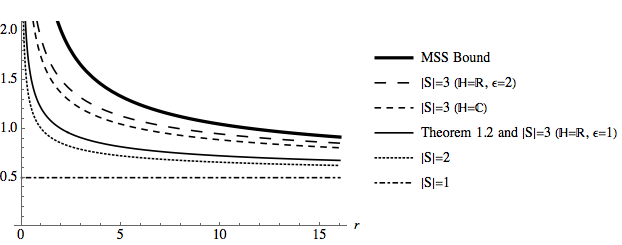}
\caption{Comparison of subset norm bounds for $\delta=1/2$.}
\label{fig2}
\end{figure}

Lastly, we note that for equiangular tight frames, numerical calculations show that subsets of frame vectors having cardinality 4 or higher do not have uniform subset norms (as Lemma \ref{le:subset-size-three} has revealed). Therefore, it is difficult to solve the corresponding system of equations to obtain exact formulas for these subset norms and prove that they satisfy the MSS bound.  This remains an open problem for subsets of arbitrary size.

\bibliographystyle{amsplain}

\end{document}